\documentclass{amsart} \usepackage{color} \usepackage[english]{babel} 
\usepackage{amsmath,amsfonts,amsthm} 
\usepackage{fancyhdr} 
\usepackage{verbatim} 
\usepackage[pdfpagelabels]{hyperref} 
\hypersetup{          
    colorlinks=true, linkcolor=blue, filecolor=magenta, urlcolor=blue,
    citecolor=magenta, linktoc=all, }

\usepackage[nameinlink]{cleveref}

\usepackage{graphicx} 
\usepackage[font=footnotesize,labelfont=bf]{caption} 
\usepackage{tikz-cd} 
\usepackage{tikz} 
\tikzset{
    MyPersp/.style={scale=1.8,x={(-2cm,-0.5cm)},y={(2cm,-0.5cm)},
        z={(0,1cm)}},
    MyPoints/.style={fill=white,draw=black,thick}
}
\usetikzlibrary{intersections}
\usetikzlibrary{external}
\usetikzlibrary{calc,3d}

\usepackage{float} 

\numberwithin{equation}{section} 
\numberwithin{figure}{section} 
\numberwithin{table}{section} 

\title[Mixed t\^ete-\`a-t\^ete twists as monodromies]{ Mixed t\^ete-\`a-t\^ete 
twists as monodromies associated with 
holomorphic function germs} 
\author{Pablo Portilla Cuadrado}
\thanks{First author supported by SVP-2013-067644 Severo Ochoa FPI grant and by 
    project by MTM2013-45710-C2-2-P, the two of them by the Spanish Ministry of 
    Economy and Competitiveness MINECO; also supported 
    by the project ERCEA 615655 NMST Consolidator Grant.}
\author{Baldur Sigur{\dh}sson} 
\thanks{The second author is supported by 
the ERCEA 615655 NMST Consolidator Grant, by the Basque Government
through the BERC 2014-2017 program and by the Spanish Ministry of Economy and
Competitiveness MINECO: BCAM Severo Ochoa excellence accreditation
SEV-2013-0323.}

\date{\normalsize\today} 

\usepackage{enumerate}
\usepackage{enumitem}

\usepackage{aliascnt}

\usepackage[all]{xy}

\usepackage{pifont}

\usepackage{setspace}
\usepackage[refpage,intoc]{nomencl}

\usepackage{tikz-cd} 

\newtheorem{thmA}{Theorem}

\newtheorem{thm}[equation]{Theorem}
\Crefname{thm}{Theorem}{thm}

\newtheorem{lemma}[equation]{Lemma}

\Crefname{prop}{Proposition}{thm}

\newtheorem{cor}[equation]{Corollary}

\theoremstyle{definition}

\newtheorem{example}[equation]{Example}

\newtheorem{definition}[equation]{Definition}

\newtheorem{block}[equation]{}

\newtheorem{notation}[equation]{Notation}
\Crefname{notation}{Notation}{notation}

\newcommand{\RR}{\mathbb{R}}
\newcommand{\QQ}{\mathbb{Q}}

\newcommand{\C}{\mathbb{C}}
\newcommand{\NN}{\mathbb{N}}

\newcommand{\ZZ}{\mathbb{Z}}

\newcommand{\MCG}{\mathrm{MCG}}
\newcommand{\rot}{\mathrm{rot}}

\newcommand{\G}{\Gamma}
\newcommand{\MS}{\mathbb{S}}
\newcommand{\TG}{\tilde{\G}}

\newcommand{\Si}{\Sigma}

\newcommand{\tat}{t\^ete-\`a-t\^ete }
\newcommand{\Tat}{T\^ete-\`a-t\^ete }

\newcommand{\NNd}{\mathcal{N}}

\newcommand{\id}{\mathop{\mathrm{id}}\nolimits}

\newcommand{\B}{\mathcal{B}}
\newcommand{\calA}{\mathcal{A}}
\newcommand{\calC}{\mathcal{C}}
\newcommand{\calD}{\mathcal{D}}
\newcommand{\calB}{\mathcal{B}}

\renewcommand{\epsilon}{\varepsilon}

\newcounter{dummy}
\renewcommand{\thedummy}{\roman{dummy}}

\newtheorem{remark}[equation]{Remark}

\newcommand{\mynd}[4]
{
    \begin{figure}[#2]\begin{center}
            \ifpdf
            \input{#1.pdf_t}
            \else
            \input{#1.pstex_t}
            \fi
            \ifthenelse{\equal{#3}{}}
            {
            }
            {
                \caption{#3\label{#4}}
            }
        \end{center}
    \end{figure}
}

\begin{document}

    \begin{abstract}
    \Tat graphs were introduced by N. A'Campo in 2010 with the goal of
    modeling the monodromy of isolated plane curves. Mixed \tat graphs
    provide a generalization which define mixed \tat twists, which are
    pseudo-periodic automorphisms on surfaces. We characterize the
    mixed \tat twists as those pseudo-periodic automorphisms that have a
    power which is a product of right-handed Dehn twists around disjoint
    simple closed curves, including all boundary components.
    It follows that the class of \tat twists coincides with that of
    monodromies associated with reduced function germs on isolated complex
    surface singularities.
    \end{abstract}

 \maketitle
 
    \tableofcontents

    \section{Introduction}
    
    In \cite{Camp1} N. A'Campo introduced the notion of pure \tat graph in 
    order to model monodromies of plane curves. These 
    are metric ribbon graphs without univalent vertices that satisfy a special 
    property. If one sees the ribbon graph $\G$ as a strong deformation 
    retract of a surface $\Si$ with boundary, the \tat propety says that 
    starting  at any point $p$ and walking a distance of $\pi$ in any direction 
    and always turning right at vertices, gets you to the same point. This 
    property defines an  element in the mapping class group $\MCG^+(\Si, 
    \partial \Si)$ which is freely periodic, and is called the
    \emph{\tat twist} associated with $\G$.

    In \cite{Graf}, C. Graf proved 
    that if one allows univalent vertices in \tat graphs, then one is able to 
    model  all freely periodic mapping classes of $\MCG^+(\Si, \partial \Si)$ 
    with positive  fractional Dehn twist coefficients. In \cite{Bob} this 
    result was improved  by showing that one does not need to enlarge the 
    original class of metric  ribbon graphs used to prove the same theorem.
    
    However, the geometric monodromy of an isolated plane curve 
    singularity with one branch and more than one Puiseux pair is not of 
    finite 
    order (see \cite{Camp2}). For this purpose, the definition of mixed \tat 
    graph was introduced in \cite{Bob}. These are 
    metric ribbon graphs together with a filtration and a set of locally 
    constant functions that model a class of pseudo-periodic automorphisms.
    It was proved that this class includes the monodromy associated with an 
    isolated singularity on a plane curve with one branch.
    
    In this work we continue the study of (relative)
    mixed  \tat graphs and we improve and generalize results from \cite{Bob}.
    The main result is a complete characterization of the mapping classes that 
    can be modeled by a mixed \tat graph. This is the content of 
    \Cref{thm:realization_mixed} which says
    
    \begin{thmA}
    Let $\phi:\Si \to \Si$ be an automorphism with fixes the boundary.  Then 
    there exists a mixed \tat graph in $\Si$ inducing its mapping class in 
    $\MCG^+(\Si, \partial \Si)$ if and only if some power of
    $\phi$ is a composition of right handed Dehn twists around
    disjoint simple closed curves including all boundary components.
    \end{thmA}

    A reduced holomorphic function germ on an isolated surface singularity
    has an associated Milnor fibration with a monodromy which fixes the
    boundary. It is known that this monodromy is pseudo-periodic and a power
    of it is a composition of right-handed Dehn twists around disjoint simple
    closed curves, which include all the boundary curves.
    It follows that the Milnor fiber associated with such a function germ
    contains a mixed \tat graph which defines the monodromy.
    Conversely, a result by Neumann and Pichon 
    \cite{NeumPich} says that any such a surface automorphism is realized
    as the monodromy associated with a function germ. Hence

    \begin{thmA}
    Mixed \tat twists are precisely the monodromies associated with reduced
    function germs on isolated surface singularities.
    \end{thmA}
    
    The structure of the work is the following. In 
    \Cref{sec:mapping_class_group} we briefly introduce notation
    related to the mapping 
    class group that we use throughout the text. 
    
    In \Cref{sec:pseudo} we fix notation and conventions about pseudo-periodic 
    automorphisms of surfaces. Not all of this section is contained in 
    \cite{Bob} since we treat a broader class of pseudo-periodic automorphisms 
    in the present text. In particular we allow amphidrome orbits of annuli. 
    This section is important for \Cref{sec:main_thm} which contains the main 
    results.
    
    In \Cref{sec:pure_tat,sec:mixed} we recall some necessary definitions and 
    results about pure and mixed \tat graphs.
    
    \Cref{sec:main_thm} starts with the statement of the main
    \Cref{thm:realization_mixed}. Two Lemmas which are used in the proof follow
    and the section ends with the proof of the main result.
    
    The work ends with an example that depicts a mixed 
    \tat graph.
      
    \section{The mapping class group}\label{sec:mapping_class_group}
    
    \begin{block}\label{bloc:fixed_bound}
        Let $\Si$ be a surface with $\partial \Si \neq \emptyset$. We denote by
        $\MCG(\Si)$ the 
        mapping class group given by the automorphisms of $\Si$ up to isotopy, 
        where 
        the automorphisms of the isotopy do not neccesarily fix the boundary. 
        Let $\partial^1 \Si \subset \partial \Si$ be a subset formed by some 
        boundary 
        components of $\Si$. We will denote by $\MCG(\Si, \partial^1 \Si)$ the 
        mapping 
        class group given by automorphisms of $\Si$ that are the identity 
        restricted 
        to $\partial^1 \Si$ and where isotopies are along automorphisms that 
        are the identity on these boundary components.
        
        Let $\phi: \Si \to \Si$ be a automorphism, we denote its class in 
        $\MCG(\Si)$ by
        $[\phi]$. If 
        $\phi|_{\partial^1 \Si}=\id$ we denote its class in $\MCG(\Si, 
        \partial^1 \Si)$
        by $[\phi]_{\partial^1 \Si}$.
        
        Given two automorphisms $\phi$ and $\psi$ of $\Sigma$ that both leave 
        invariant some subset $B\subset \partial \Si$ such that 
        $\phi|_B=\psi|_B$, we say 
        they are \emph{isotopic relative to the action $\phi|_B$} if there 
        exists a family of automorphisms of $\Sigma$ that isotope them as 
        before and such that any 
        automorphism of the family has the same restriction to $B$ as $\phi$ 
        and $\psi$. We write   $[\phi]_{B, \phi|_B}=[\psi]_{B,\phi|_B}$. We 
        denote by $\MCG(\Si, B, \phi|_B)$ the set of classes 
        $[\phi]_{B,\phi|_B}$ with respect to this equivalence relation. We 
        denote by $\MCG^+(\Si, B, \phi|_B)$ if we restrict to automorphisms 
        preserving orientation.
    \end{block}
    
    \begin{block}\label{bloc:rotation_numbers}
        Consider an automorphism $\phi: \Si \to \Si$ with $\phi|_{\partial^1 
            \Si}= \id$ 
        for some subset $\partial^1 \Si \subset \partial \Si$. Let $\calD_i$ 
        denote a 
        right-handed Dehn twist around a curve parallel to the boundary 
        component $C_i 
        \subset \partial ^1\Si $. Suppose that $[\phi] \in \MCG(\Si)$ is of 
        finite order, 
        i.e. there exists a natural number  $n$ such that $[\phi]^n = [\id]$. 
        Then we have that $[\phi^n]_{\partial^1 \Si} = [\calD_1 ]_{\partial^1 
        \Si}^{t_1} \cdots [\calD_r]_{\partial^1 \Si}^{t_r}$ with $t_i \in \ZZ$. 
        We call $t_i/n$ the {\em fractional Dehn twist coefficient} of $\phi$ 
        at the component 
        $C_i$. 
    \end{block}

    \section{Pseudo-periodic automorphisms}\label{sec:pseudo}
    
    We recall conventions, definitions and results from \cite[Section 6.]{Bob}. 
    that we 
    will use in the present work. We also extend some of the notions there to 
    cover some cases that were not treated in that work.
    
    \begin{definition}\label{def:pseudo}
        A automorphism $\phi:\Si \rightarrow \Si$ is pseudo-periodic if it is 
        isotopic 
        to a automorphism satisfying that there exists a finite collection of 
        disjoint 
        simple closed curves $\mathcal{C}$ such that 
        \begin{enumerate}
            \item $\phi(\mathcal{C})= \mathcal{C}$.
            \item $\phi|_{\Si \setminus \mathcal{C}}$ is freely isotopic to a 
            periodic automorphism.
        \end{enumerate}
        
        Assuming that none of the connected components of $\Si \setminus 
        \mathcal{C}$ is either a disk or an annulus and that the set of curves 
        is minimal, 
        which is 
        always possible, we name $\calC$ an admissible set of curves for 
        $\phi$.  
    \end{definition}

    The following theorem is a particularization on pseudo-periodic 
    automorphisms 
    of the more general Corollary 13.3 in \cite{Farb} that describes a {\em 
        canonical form} for every automorphism of a surface. 
    
    \begin{thm}[Almost-Canonical Form and Canonical Form]\label{theo:canonical} 
        Let 
        $\Si$ be a surface with $\partial\Si\neq \emptyset$. Any 
        pseudo-periodic map of 
        $\Si$ is isotopic to an automorphism in \emph{almost-canonical form}, 
        that 
        means a automorphism $\phi$ which has an  admissible set of curves 
        $\calC=\{\calC_i\}$ and annular neighborhoods $\calA=\{\calA_i\}$ with 
        $\calC_i \subset \calA_i$ such that
        \begin{enumerate}
            \item $\phi(\calA)=\calA$. 
            \item The map $\phi|_{\overline{\Si\setminus\calA}}$ is periodic.
        \end{enumerate}
        When the set $\calC$ is minimal we say that $\phi$ is in canonical form.
    \end{thm}

    \begin{remark}\label{rem:mod_canonical_form}
        In the case we have a  pseudo-periodic automorphism  of $\Si$ that 
        fixes 
        pointwise some components $\partial^1\Si$ of the boundary $\partial 
        \Si$ we can 
        always find \emph{a canonical form} as follows. We can find an isotopic 
        automorphism $\phi$ relative to $\partial^1 \Si$ that coincides with a 
        canonical form as in the previous theorem outside a collar neighborhood 
        $U$ of 
        $\partial^1 \Si$. We may assume that there exists an isotopy connecting 
        the 
        automorphism and its canonical form relative to $\partial^1\Si$.
    \end{remark}
    
    \begin{block}\label{bloc:amphidrome_curves}
        Let $\phi$ be a pseudo-periodic automorphism in some almost-canonical 
        form. 
        Let $\calC_1, \ldots, \calC_k$ be a subset of curves in $\calC$ that 
        are 
        cyclically permuted by $\phi$, i.e. $\phi(\calC_i)= \calC_{i+1}$ for 
        $i=1,\ldots, k-1$ and $\phi(\calC_k)=C_1$. Suppose that we give an 
        orientation 
        to $\calC_1, \ldots, \calC_k$ so that $\phi|_{\calC_i}$ for $i=1, 
        \ldots, k-1$ 
        is orientation preserving. We say that the curves are {\em amphidrome} 
        if $\phi|_{\calC_k}:\calC_k \to \calC_1$ is orientation reversing. 
    \end{block}
    
    \begin{notation}\label{not:D_an}
        Let $s,c \in \RR$. We denote by $\calD_{s,c}$ the automorphism of 
        $\MS^1 
        \times I$ given by by $(x,t) \mapsto (x+st+c,t)$ (we are taking $\MS^1 
        = \RR/ 
        \ZZ$). Observe that 
        $$\calD_{s,c}\circ \calD_{s',c'} = \calD_{s+s',c+c'},$$
        $$\calD_{s,c}^{-1}=\calD_{-s,-c}.$$
        In this work we always have $s \in \QQ$.
    \end{notation}
    \begin{remark}\label{rem:linear_prep}
        We can isotope
        $\calD_{s,c}$ to a automorphism $\calD^p_{s,c}$ that is periodic on 
        some tubular
        neighborhood of the core 
        curve $\MS^1 \times \{1/2\}$ of the annulus while preserving the action 
        on 
        the boundary $\partial (\MS^1 \times I)$:
        
        \begin{equation}
        \calD^p_{s,c}(x,t)=\left\{
        \begin{array}{ll}
        (x+  3s(t- \frac{1}{3}) + c,t) & 0 \leq t \leq \frac{1}{3} 
        \vspace*{0.1cm} \\
        (x+ c, t)  & \frac{1}{3} \leq t \leq \frac{2}{3}  \vspace*{0.1cm} \\
        (x+3s(t- \frac{2}{3})+c,t)  & \frac{2}{3} \leq t \leq 1
        \end{array}
        \right.
        \end{equation}
    \end{remark}
    \begin{notation}\label{not:amphi_dehn}
        We denote by $\tilde{\calD}_{s}$ the automorphism of $\MS^1 \times I$ 
        given by 
        
        \begin{equation}
        \tilde{\calD}_{s}(x,t)=\left\{
        \begin{array}{ll}
        (-x-3s(t- \frac{1}{3}),1-t) & 0 \leq t \leq \frac{1}{3} \vspace*{0.1cm} 
        \\
        (-x, 1-t)  & \frac{1}{3} \leq t \leq \frac{2}{3}  \vspace*{0.1cm} \\
        (-x-3s(t- \frac{2}{3}),1-t)  & \frac{2}{3} \leq t \leq 1
        \end{array}
        \right.
        \end{equation}
        In this case we only work with $s \in \QQ$ as well.
    \end{notation}
    
    \begin{definition}\label{def:dehn_twist}
        Let $\calC \subset \Si$ be a simple closed curve embedded in an 
        oriented surface
        $\Si$. And let $\calA$ be a tubular neighborhood of $\calC$. Let 
        $\calD: \Si \to
        \Si$ be a automorphism of the surface with
        $\calD|_{\Si \setminus \calA} = \id$.
        We say that $\calD$ is a negative Dehn twist around $\calC$ or a 
        right-handed
        Dehn twist if there exist a parametrization $\eta: \MS^1 \times I 
        \rightarrow \calA$ preserving orientation  such that $$\calD = \eta 
        \circ 
        \calD_{1,0}\circ \eta^{-1}.$$
        
        A positive Dehn twist is defined similarly changing $\calD_{1,0}$ by 
        $\calD_{-1,0}$ in the formula above.
    \end{definition}

    \begin{lemma}[Linearization. Equivalent to Lemma 2.1 in 
        \cite{Ami0}]\label{lem:lin} 
        Let $\calA_i$ be an annulus and let $\phi: \calA_i \to \calA_i$ be a 
        automorphism that does not exchange boundary components. Suppose that 
        $\phi|_{\partial \calA_i}$ is periodic.
        Then, after an isotopy fixing the boundary,
        there exists a parametrization 
        $\eta:\MS^1 \times I\to \calA_i$ such that $$\phi=\eta\circ 
        \calD_{-s,-c}\circ\eta^{-1}$$ for some $s\in \QQ$, some $c\in \RR$.
    \end{lemma}
    
    \begin{lemma}[Specialization.  Equivalent to Lemma 2.3 in 
        \cite{Ami0}]\label{lem:spe} 
        Let $\calA_i$ be an annulus and let $\phi: \calA_i \to \calA_i$ be a 
        automorphism that exchanges boundary components. Suppose that 
        $\phi|_{\partial \calA_i}$ is periodic.
        Then after an isotopy fixing the boundary,
        there exists a parametrization $\eta:\MS^1 \times 
        I\to \calA_i$ such that $$\phi=\eta\circ 
        \tilde{\calD}_{-s}\circ\eta^{-1}$$ for 
        some $s\in \QQ$.
    \end{lemma}
    
    \begin{remark}\label{re:lin} In the case $\phi|_{\partial \calA_i}$ is the 
        identity, we have that $$\phi= \eta\circ \calD_{s,0}\circ\eta^{-1}$$
        for some $s\in \ZZ$, that is $\phi=\calD^s$ for some right-handed Dehn 
        twist as
        in Definition 
        \ref{def:dehn_twist}. 
    \end{remark}

    \begin{definition}[Screw number]\label{def:screw_number}
        Let $\phi$ be a pseudo-periodic automorphism as in 
        \Cref{theo:canonical}. Let 
        $n$ be the order of $\phi|_{\Si \setminus \calA}$. By \Cref{re:lin}, 
        $\phi^n|_{\calA_i}$ equals $\calD|_{\calA_i}^{s_i}$ for a certain $s_i 
        \in \ZZ$.
        
        Let $\alpha$ be the length of the orbit in which $\calA_i$ lies and let 
        $\tilde{\alpha}\in\{\alpha, 2\alpha\}$ be the smallest number such that 
        $\phi^{\tilde{\alpha}}$ does not exchange the boundary components of 
        $\calA_1$. 
        We define $$s(\calA_i):=\frac{-s_i}{n} \tilde\alpha.$$ We call 
        $s(\calA_i)$ the 
        {\em screw number} of $\phi$ at $\calA_i$ or at $\calC_i$.
    \end{definition}
    
    \begin{remark}\label{rem:screw_indep}
        Our \Cref{def:screw_number} coincides with \cite[Definition 2.4]{Ami0}. The original definition is due to Nielsen [\cite{Niel}. 
        Sect. 
        12] and it does not depend on a canonical form for $\phi$. Since we are 
        restricting to automorphisms that do not exchange boundary components 
        of the 
        annuli $\calA$, our definition is a bit simpler. 
    \end{remark}

    \begin{lemma}\label{lem:linear_ex} Let $\phi$ be a automorphism as in 
        \Cref{theo:canonical} and let  $\{\calA_i\} \subset \calA$ be a set of 
        $k$ 
        annuli cyclically permuted by $\phi$, i.e. $\phi(\calA_i) = 
        \calA_{i+1}$ such 
        that $\phi^k$ does not exchange boundary components. Then there exist 
        coordinates $$\eta_{i}:\MS^1 \times I\to \calA_i$$ for the annuli in 
        the orbit 
        such that 
        $$ \eta_{j+1}^{-1}\circ\phi\circ\eta_{j}=\calD_{-s/k,-c/k} $$ where $s$ 
        and $c$ 
        are associated to $\calA_1$ as in \Cref{lem:lin}. 
    \end{lemma}
    \begin{proof}
        See \cite[Lemma 6.17]{Bob}.
    \end{proof}
    
    \begin{remark}\label{rem:mod_linear_ex}
        By  \Cref{rem:linear_prep} we can substitute $\calD_{-s/k,-c/k}$ by
        $\calD^p_{-s/k,-c/k}$ in the previous lemma.
    \end{remark}

    \begin{lemma}\label{lem:special_ex} Let $\phi$ be a automorphism as in 
        \Cref{theo:canonical} and let  $\{\calA_i\} \subset \calA$ be a set of 
        $k$ 
        annuli cyclically permuted by $\phi$, i.e. $\phi(\calA_i) = 
        \calA_{i+1}$ such 
        that $\phi^k$ exchanges boundary components. Then there exist 
        coordinates 
        $$\eta_{i}:\MS^1 \times I\to \calA_i$$ for the annuli in the orbit such 
        that 
        $$ \eta_{j+1}^{-1}\circ\phi\circ\eta_{j}=\tilde{\calD}_{-s/\alpha} $$ 
        where $s$ 
        is associated to $\calA_1$ as in \Cref{lem:spe}. 
    \end{lemma}
    
    \begin{proof}
        Take a parametrization of $\calA_1$ for $\phi^k:\calA_1 
        \rightarrow \calA_1$ as in \Cref{lem:spe}, say  $\eta_1:\MS^1 \times 
        I\to \calA_1$.
        Define recursively $\eta_j:=\phi\circ\eta_{j-1}\circ 
        \tilde{\calD}_{s/k}$ (see 
        \Cref{not:D_an}).
        Then, we have
        $$\eta_{j+1}^{-1}\circ\phi\circ\eta_{j}=\tilde{\calD}_{-s/k}.$$
        Since for every $j$ we have that $\eta_j=\phi^{j-1}\circ\eta_1\circ 
        \tilde{\calD}_{s(j-1)/k}$ we have also that 
        $$\eta_1^{-1}\circ\phi\circ 
        \eta_{\alpha}=\eta_1^{-1}\circ\phi\circ\phi^{k-1}\circ\eta_1\circ 
        \tilde{\calD}_{s(k-1)/k}=\tilde{\calD}_{-s/k}.$$
    \end{proof}
    \begin{remark}\label{re:sk} 
        After this proof we can check that $\eta_k^{-1} 
        \circ\phi^\alpha\circ\eta_k=\calD_{-s,-c}$ to see that the screw number 
        $s=s(\calA_i)$ and the parameter $c$ modulo $\ZZ$ of 
        \Cref{lem:linear_ex} only 
        depend on the orbit of $\calA_i$. 
        
        We observe also that the numbers $s$ and $c$ of \Cref{lem:linear_ex} 
        satisfy
        \begin{itemize}
            \item $s$ equals $s(\calA_i)$ and 
            \item $c$ is only determined modulo $\ZZ$ and equals the usual 
            rotation number
            $\rot (\phi^{\alpha_i}|_{\eta(\MS^1 \times \{0\})})\in (0,1]$. 
        \end{itemize} 
        This is also observed in Corollary 2.2 in \cite{Ami0}.
    \end{remark}

    \begin{definition}\label{def:boundary_dehn}
        Let $C$ be a component of $\partial \Si$ and let $\calA$ be a compact 
        collar 
        neighborhood of $C$ in $\Si$. Suppose that $C$ has a metric and total 
        length  
        is equal to $\ell$. Let $\eta:\MS^1 \times I\to \calA$ be a 
        parametrization of
        $\calA$, 
        such that  $\eta|_{\MS^1 \times \{1\}}:\MS^1 \times \{1\} \to C$ is an 
        isometry.
        
        Suppose that $\MS^1$ has the metric induced from taking $\MS^1 = \RR/ 
        \ell\ZZ$ 
        with $\ell \in \RR_{>0}$ and the standard metric on $\RR$. A 
        \emph{boundary Dehn twist} of length $r \in \RR_{>0}$ along $C$ is a 
        automorphism 
        $\calD^\eta_{r} 
        (C)$ of $\Si$ such that:
        \begin{enumerate}
            \item it is the identity outside $\calA$
            \item the restriction of $\calD^\eta_{r} (C)$ to $\calA$ in the 
            coordinates
            given by 
            $\eta$ is given by $(x,t) \mapsto ( x+ r\cdot t, t).$
        \end{enumerate}
        The isotopy type of $\calD^\eta_{r} (C)$ by isotopies fixing the action 
        on 
        $\partial \Si$ does not depend on the parametrization $\eta$. When we 
        write just $\calD_{r}(C)$, it means that we are considering a boundary 
        Dehn twist 
        with 
        respect to {\em some} parametrization $\eta$.
    \end{definition}
    
    \begin{remark}\label{re:g1}
        Given a automorphism $\phi$ of a surface $\Si$ with $\partial \Si\neq 
        \emptyset$. Let $C$ be a connected component of $\partial \Si$ such that
        $\phi|_{C}$ is a
        rotation by $c \in [0,1)$. Let $\calA$ be a 
        compact collar neighborhood of $C$ (isomorphic to $I\times C$) in 
        $\Si$.  Let 
        $\eta:\MS^1 \times I\to \calA$ be a parametrization of $\calA$, with 
        $\phi(\MS^1
        \times 
        \{1\})=C$. Up to isotopy, we can assume that the restriction of $\phi$ 
        to
        $\calA$ satisfies 
        $$
        \eta^{-1}\circ\phi|_{\calA}\circ\eta(x,t)=(x+c,t).
        $$ 
    \end{remark}

    \section{Pure \tat graphs}\label{sec:pure_tat}
    
    In this section we recall some definitions and conventions from \cite{Bob}.

    \begin{block}
        A graph $\G$ is a $1$ dimensional finite CW-complex; unless otherwise 
        specified a graph doesn't have univalent vertices. A ribbon graph is a 
        graph equipped with a cyclic ordering of the edges adjacent to each 
        vertex.    With a ribbon graph, 
        one can recover the topology of an orientable surface with boundary, we 
        call this surface the {\em thickening} of $\G$. A metric ribbon graph 
        is a 
        ribbon 
        graph with a metric on each of its edges.
        
        A relative metric ribbon graph (see \cite[Definition 2.3]{Bob} for the notion of relative ribbon graph and beginning of Section 3. therein) is a pair $(\G,A)$ with $A \subset \G$ a 
        subgraph formed by a disjoint union of circles with the property that 
        for each connected 
        component $A_i \subset A$, there exists a boundary component on the 
        thickening 
        $\Si$ of $\G$ such that it retracts to $A_i$. The relative thickening 
        of $(\G,A)$ is the thickening of $\G$ minus the cylinders corresponding 
        to the  boundary components that retract to $A$. In particular, the 
        relative thickening, also denoted by $\Si$ contains $A$ as boundary.
    \end{block}

    \begin{definition}(\cite[Definition 3.2]{Bob} Safe walk)\label{def:safe}
        Let $(\G,A)$ be a metric relative ribbon graph. A safe walk for a point 
        $p$ in 
        the interior of some edge is a path
        $\gamma_p:\RR_{\geq 0} \rightarrow \Gamma$ with $\gamma_p(0)=p$ and 
        such that:
        
        \begin{enumerate}
            \item[(1)] The absolute value of the speed $|\gamma_p'|$ measured 
            with the 
            metric of $\G$ is constant and equal to  $1$. Equivalently, the 
            safe walk is 
            parametrized by arc length, 
            i.e. for $s$ small enough $d(p,\gamma_p(s))=s$. 
            \item[(2)] when $\gamma_p$ gets to a vertex, it continues along the 
            next edge in the given cyclic order.
            \item[(3)] If $p$ is in an edge of $A$, the walk $\gamma_p$ starts 
            running in 
            the opposite direction to the one indicated by $A$ seen as boundary 
            of $\Si$.
        \end{enumerate}
        An $\ell$-safe walk is the restriction of a safe walk to the interval 
        $[0,\ell]$.
        If a length is not specified when referring to a safe walk, we will 
        understand 
        that its length is $\pi$.
    \end{definition}
    
    The notion in (2) of \emph{continuing along the next edge in the order} of 
    $e(v)$ is equivalent to the notion of \emph{turning to the right} in every 
    vertex for paths
    parallel to $\G$ in $\Si$ in A'Campo's words in \cite{Camp1}.
    
    \begin{definition}(\cite[Definition 3.5]{Bob} \Tat property) \label{def:tat}
        Let $(\G,A)$ be relative metric ribbon graph. 
        We say that $\G$ satisfies the \emph{$\ell$-\tat property}, or that 
        $\G$ is an 
        \emph{$\ell$-\tat graph} if
        
        \begin{enumerate}[label=(\arabic*)]
            \item For any point $p\in \G\setminus (A\cup v(\G))$ the two 
            different 
            $\ell$-safe walks starting at $p$, that we denote 
            by $\gamma_p, \omega_p$, satisfy $\gamma_p(\ell)=\omega_p(\ell)$.
            \item for a point $p$ in $A\setminus v(\G)$, the end point of the 
            unique 
            $\ell$-safe walk starting at $p$ belongs to $A$.
        \end{enumerate}
        
        In this case, we say that $(\G,A)$ is a \emph{relative $\ell$-\tat 
        graph}. If $A=\emptyset$, we call it a \emph{pure} $\ell$-\tat 
        structure or graph. 
        If $\ell=\pi$ we just call it pure \tat structure or graph.
    \end{definition}
    
    \begin{figure}[H]
        \includegraphics[width=80mm]{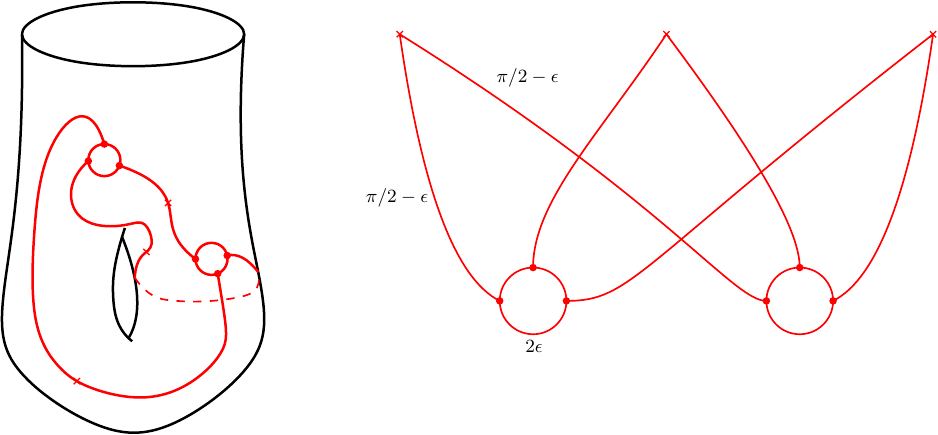}%
        \caption{An example of a relative \tat graph. It has $2$ connected 
            components in $A$ (the two small circles). The length of an edge in 
            $A$ is $2\epsilon$ and the
            length of and edge from a vertex in $A$ to a vertex depicted as a 
            cross is 
            $\pi/2 - \epsilon$.}
        \label{fig:example_relative_tat}
    \end{figure}
    
    \begin{notation}\label{not:cylinders_notation}
        Let $(\G,A) \hookrightarrow (\Si, \partial \Si)$ be a relative ribbon 
        graph 
        properly embedded in its thickening.  Let $\Si_{\G} $ be the surface 
        that
        results 
        from cutting $\Si$ along $\G$, then $\Si_{\G}$ consists of as many 
        cylinders as 
        there are connected components in $\partial \Si \setminus A$. We denote 
        these 
        cylinders by $\tilde\Si_1, \ldots, \tilde\Si_r$.
        
        Let $g_\G: \Si_\G \rightarrow \Si$
        be the gluing map.  We denote by $\widetilde \G_i$ the boundary 
        component of the cylinder $\Si_i$ that comes 
        from the graph (that is $g_\G(\widetilde \G_i)\subset \G$) and by $C_i$ 
        the one 
        coming from a boundary component of $\Si$ (that is $g_\G(C_i)\subset 
        \partial 
        \Si$). From now on, we take the convention that $C_i$ is identified 
        with $C_i 
        \times \{1\}$ and that $\TG_i$ is identified with $C_{i} \times \{0\}$.
        We set $\Si_i:=g_\G(\widetilde{\Si}_i)$ and 
        $\G_i:=g_\G(\widetilde{\G}_i)$. 
        Finally we denote $g_\G(C_i)$ also by $C_i$ since $g_\G|_{C_i}$ is 
        bijective.
        
        A {\em retraction} or a {\em product structure} for a component $\Si_i$ 
        is a 
        parametrization $$r_i: \MS^1 \times I \rightarrow \Si_i.$$ For each 
        $\theta \in 
        \MS^1$, we call $r_i(\{\theta\} \times I)$ a {\em retraction line.} We 
        also say 
        that $g_\G(r_i(\{\theta\} \times I))$ is a retraction line.
    \end{notation}

    \begin{block}\label{bloc:tat_automorphism} The following is a particular case of \cite[Definition 5.4]{Bob}
      A relative \tat graph $(\G,A) \hookrightarrow (\Si, \partial A)$ 
      has an associated mapping class $[\phi_\G]_{\partial \Si \setminus A, \id}$ on $\Si$, 
       more  specifically, an element of $\MCG(\Si, \partial \Si \setminus A)$. If a 
       product  structure is specified, then an explicit representative $\phi_\G$ is 
        induced.  For any product structure $\phi_\G$ satisfies:
        
        \begin{enumerate}[label=(\arabic*)]
            \item $\phi_\G|_{\G}(p) = \gamma_p(\pi)$, that is, it induces on 
            the graph the 
            same action as the \tat property.
            
            \item \label{it:prop_ii} the mapping class $[\phi_\G] \in 
            \MCG(\Si)$ is of
            finite order, we also say it is periodic.
            
            \item \label{it:prop_iii} the fractional Dehn twist coefficients 
            $t_i/n$ (recall \ref{bloc:rotation_numbers}) along all boundary 
            components in $\Si 
            \setminus A$ 
            are strictly positive.
        \end{enumerate}
        
        Actually, in \cite{Bob} it is proven that \ref{it:prop_ii} and 
        \ref{it:prop_iii} above characterize \tat automorphisms, more 
        concretely the following is proven:
        
        \begin{thm}\label{thm:characterization_tat} 
            Let $\phi: \Si \to \Si$ be an automorphism of a surface with 
            $\phi|_{\partial^1 \Si} = \id$ for some non-empty union of components    
            $\partial^1 \Si \subset \partial \Si$.
            Then there exists a relative \tat graph $(\G, \partial \Si 
            \setminus \partial^1 \Si)$
            with $[\phi_\G]_{\partial^1 \Si} = [\phi]_{\partial^1 \Si}$ if and 
            only if 
            $[\phi] \in \MCG(\Si)$ is of finite order and all the fractional 
            Dehn twists at 
            boundary components in $\partial^1 \Si$ are strictly positive.
        \end{thm}
        
        This theorem is a particular case of Theorem 5.27 (see also  5.14 therein) in \cite{Bob}. Actually it is a 
        consequence of 
        that Theorem since there, the authors consider also negative fractional 
        Dehn twists and also {\em negative} safe walks (which turn left instead 
        of turning right). This is done via a {\em sign} map $\iota: \partial^1 
        \Si \to \{+,-,0\}$.
        
        In this work we only use the original notion of A'Campo, so this map is 
        constant $+$.
        
    \end{block}

    \section{Mixed \tat graphs}\label{sec:mixed}

    With pure, relative and general \tat graphs we model periodic 
    automorphisms. 
    Now we extend the notion of \tat graph to be able to model some 
    pseudo-periodic 
    automorphisms.
    
    Let $(\G^\bullet, A^{\bullet})$ be a decreasing filtration on a connected 
    relative metric ribbon graph $(\G, A)$. That is
    $$(\G,A)=(\G^0, A^0) \supset (\G^1, A^1) \supset \cdots \supset (\G^d, 
    A^d)$$
    where $\supset$ between pairs means $\G^i \supset \G^{i+1}$ and $A^i 
    \supset 
    A^{i+1}$, and where $(\G^i, A^i)$ is a (possibly disconnected) relative 
    metric 
    ribbon graph for each $i=0,\ldots, d$. We say that $d$ is the depth of the 
    filtration $\G^\bullet$. We assume each $\G^i$ does not have univalent 
    vertices 
    and is a subgraph of $\G$ in the usual terminology in Graph Theory. We 
    observe 
    that since each $(\G^i, A^i)$ is a relative metric ribbon graph, we have 
    that 
    $A^i \setminus A^{i+1}$ is a disjoint union of connected components 
    homeomorphic to $\MS^1$.
    
    For each $i=0, \ldots, d$, let $$\delta_i: \G^{i} \rightarrow \RR_{\geq 
    0}$$ be a locally constant map (so it is a map constant on each connected 
    component). We put the restriction that $\delta_0(\G^0)>0$. We denote the 
    collection of all these maps by $\delta_\bullet$.
    
    Let $p \in \G$, we define $c_p$ as the largest natural number such that 
    $p\in 
    \G^{c_p}$.
    
    \begin{definition}(\cite[Definition 7.11]{Bob} Mixed safe walk)\label{def:mixed_tat} Let $(\G^\bullet, 
        A^{\bullet})$ be a filtered relative metric ribbon graph. 
        Let $p \in \G\setminus A \setminus v(\G)$. We define a mixed safe walk 
        $\gamma_p$ starting at $p$ as  a concatenation of paths defined 
        iteratively by the following properties
        \begin{itemize}
            \item[i)] $\gamma_p^0$ is a safe walk of length $\delta_0(p)$ 
            starting at 
            $p^\gamma_0:=p$. Let $p^\gamma_1:=\gamma^0(\delta_0)$ be its 
            endpoint.
            \item[ii)] Suppose that $\gamma_p^{i-1}$ is defined and let 
            $p^\gamma_i$ be its 
            endpoint. 
            \begin{itemize}
                \item If $i > c_p$ or $p^\gamma_{i} \notin \G^{i}$ we stop the 
                algorithm. 
                \item If $i\leq c_p$ and $p^\gamma_{i} \in \G^{i}$ then define 
                $\gamma_p^{i}:[0,\delta_i(p_i)] \rightarrow \G^i$ to be a safe 
                walk of length 
                $\delta_i(p^\gamma_i)$ starting at $p^\gamma_i$ and going in 
                the same direction 
                as $\gamma_p^{i-1}$.
            \end{itemize} 
            \item[iii)] Repeat step $ii)$ until algorithm stops.
        \end{itemize}
        
        Finally, define $\gamma_p:=\gamma_p^k \star \cdots \star \gamma_p^0$, 
        that is, the mixed safe walk starting at $p$ is the concatenation of 
        all the safe walks defined in the inductive process above. 
    \end{definition}
    
    As in the pure case, there are two safe walks starting at each point on $\G 
    \setminus (A \cup v(\G))$. We denote them by $\gamma_p$ and $\omega_p$.
    
    \begin{definition}(\cite[Definition 7.12]{Bob} Boundary mixed safe walk)\label{def:boundary_mixed_tat}
        Let $(\G^\bullet, A^{\bullet})$ be a filtered relative metric ribbon 
        graph and let $p \in A$. We define a boundary mixed safe walk $b_p$ 
        starting at $p$ as  a concatenation of a collection of paths defined 
        iteratively by the following properties
        \begin{itemize}
            \item[i)] $b_{p_0}^0$ is a boundary safe walk of length 
            $\delta_0(p)$ starting 
            at $p_0:=p$ and going in the direction indicated by $A$ (as in the 
            relative \tat case). Let $p_1:=b_p^0(\delta_0)$ be its endpoint.
            \item[ii)] Suppose that $b_{p_{i-1}}^{i-1}$ is defined and let 
            $p_i$ be its endpoint. 
            \begin{itemize}
                \item If $i >c_p$ or $p_{i} \notin \G^{i}$ we stop the 
                algorithm. 
                \item If $i\leq c(p)$ and $p_{i} \in \G^{i}$ then define 
                $b_{p_i}^{i}:[0,\delta_i(p_i)] \rightarrow \G^i$ to be a safe 
                walk of length $\delta_i(p_i)$ starting at $p_i$ and going in  
                the same direction as $b_{p_{i-1}}^{i-1}$.
            \end{itemize} 
            \item[iii)] Repeat step $ii)$ until algorithm stops.
        \end{itemize}
        
        Finally, define $b_p:=b_{p_k}^k \star \cdots \star b_{p_0}^0$, that is, 
        the boundary mixed safe walk starting at $p$ is the concatenation of 
        all the safe walks defined in the inductive process. 
    \end{definition}

    \begin{notation}
        We call the number $k$ in \Cref{def:mixed_tat} (resp. 
        \Cref{def:boundary_mixed_tat}), the \textit{order} of the mixed safe 
        walk (resp.boundary mixed safe walk) and denote it by $o(\gamma_p)$ 
        (resp. $o(b_p)$). 
        
        We denote by  $l(\gamma_p)$  the \textit{length} of the mixed safe walk 
        $\gamma_p$ which is the sum $\sum_{j=0}^{o(\gamma_p)} 
        \delta_j(p^\gamma_j)$ of 
        the lengths of all the walks involved. We consider the analogous 
        definition 
        $l(b_p)$.
        
        As in the pure case, two mixed safe walks starting at $p \in \G 
        \setminus v(\G)$ exist. We denote by $\omega_p$ the mixed safe walk 
        that starts at $p$ but in the opposite direction to the starting 
        direction of $\gamma_p$.
        
        Observe that since the safe walk $b_{p_0}^0$ is completely determined 
        by $p$, for a point in $A$ there exists only one boundary safe walk.
    \end{notation}
    
    Now we define the relative mixed \tat property.
    
    \begin{definition}(\cite[Definition 7.14]{Bob}Relative mixed \tat property)
        \label{def:relative_mix_sim}
        Let $(\G^\bullet, A^{\bullet})$ be a filtered relative metric ribbon 
        graph and 
        let $\delta_\bullet$ be a set of locally constant mappings $\delta_k: 
        \G^k \rightarrow \RR_{\geq 0}$. We say that  $(\G^{\bullet}, A^\bullet, 
        \delta_\bullet)$ satisfies the relative mixed \tat property or that it 
        is a relative mixed \tat graph if for every $p \in \G-(v(\G) \cup A)$
        \begin{itemize}
            \item[I)] The endpoints of $\gamma_p$ and $\omega_p$ coincide.
            \item[II)] $c_{\gamma_p(l({\gamma_p}))}=c_p$
        \end{itemize}
        and for every $p \in A$, we have that  
        \begin{itemize}
            \item[III)] $b_p(l({b_p})) \in A^{c_p}$
        \end{itemize} 
    \end{definition}
    
    As a consequence of the two previous definitions we have:
    \begin{lemma} \label{lem:mix_prop}
        Let $(\G^{\bullet}, A^\bullet,\delta_\bullet)$ be a mixed relative \tat 
        graph, then
        \begin{itemize}
            \item[a)] $o(\omega_p)=o(\gamma_p)= c_p$
            \item[b)]$l(\gamma_p) = l(\omega_p)$ for every $p \in \G \backslash 
            v(\G)$.
        \end{itemize}
    \end{lemma}
    
    \begin{proof}
    	See \cite[Lemma 7.15]{Bob}.
    \end{proof}
    \begin{remark} Note that for mixed \tat graphs it is not true that 
        $p\mapsto 
        \gamma_p(\delta(p))$ gives a continuous mapping from $\G$ to $\G$.
    \end{remark}
    
    \begin{remark}\label{rem:equiv_mixed_tat}
        Satisfying $I)$ and $II)$ of the mixed \tat property in 
        \Cref{def:relative_mix_sim} is equivalent to satisfying:
        \begin{itemize}
            \item[I')] For all $i=0, \dots, d-1$, the automorphism $\widetilde
            \phi_{\G,i}=\calD_{\delta_i}\circ\phi_{\G,i-1}$ is compatible with 
            the gluing
            $g_i$, that is, $$g_i(x) =g_i(y) \Rightarrow 
            g_i(\widetilde{\phi}_{\G,i}(x)) =
            g_i(\widetilde{\phi}_{\G,i}(y)).$$
        \end{itemize}
        Below we see the diagram which shows the construction of $\phi_\G$.

        \begin{equation}
        \begin{tikzcd}\label{diag:mixed_tat}
        \Si_{\G^{0}} \arrow[rightarrow]{r}{ \phi_{\G,-1}} \arrow{d}{g_0} &
        \Si_{\G^{0}} \arrow[rightarrow]{r}{\calD_{\delta_0}}             &
        \Si_{\G^{0}}  \arrow{d}{g_0}                                     &
        & 
        & 
        & 
        &\\
        \Si_{\G^{1}} \arrow{d}{g_1} \arrow[rightarrow]{rr}{ \phi_{\G,0}} &
        &
        \Si_{\G^{1}} \arrow[rightarrow]{r}{\calD_{\delta_1}} &
        \Si_{\G^{1}} \arrow{d}{g_1}  &  &  & &\\
        \Si_{\G^2} \arrow{d}{g_2} \arrow{rrr}{ {\phi}_{\G,1}}           &  &  &
        \Si_{\G^{2}} \arrow{r}{\calD_{\delta_2}}                      &
        \Si_{\G^{2}}  \arrow{d}{g_2}                                 &   &&&\\
        \vdots                                                           &
        \vdots                                                          &
        \vdots                                                         &
        \vdots                                                       &
        \vdots                                                      &  &    & \\
        \Si_{\G^d} \arrow{d}{g_d} \arrow{rrrrrr}{ {\phi}_{\G,d-1}}       & &&&&&
        \Si_{\G^d} \arrow{r}{\calD_{\delta_d}}                           &
        \Si_{\G^d} \arrow{d}{g_d}                                           \\
        \Si \arrow{rrrrrrr}{\phi_{\G}=\phi_{\G,d}}                       &&&&&&&
        \Si                                                             
        \end{tikzcd}
        \end{equation}
     \end{remark}
       \begin{remark}\label{rem:screw_number}
        The pseudo-periodic automorphism $\phi_\G$ induced by a mixed \tat 
        graph has negative screw numbers and positive fractional Dehn twist 
        coefficients  as noted in \cite[Remark 7.28]{Bob}. Actually, it was also proved in 
        \cite{Bob} that the screw number 
        associated to an orbit of annuli  $\calA^i_{j,1}, \ldots, 
        \calA^i_{j,k}$ 
        between levels $i-1$  and $i$ of the filtration is            
        \begin{equation}
            -\sum_k \delta_i(\G^i_{j,k}) / l(\TG^i_{j,1}).
          \end{equation}
        \end{remark}
    
    \section{Realization theorem} \label{sec:main_thm}
    
    In this section we prove \Cref{thm:realization_mixed} which is the main 
    results of this paper. It characterizes the pseudo-periodic automorphisms 
    that can be realized by mixed \tat graphs. First we introduce some notation 
    and conventions.

    Let $\phi:\Si \rightarrow \Si$ be a pseudo-periodic automorphism. For the 
    remaining of this work we impose the following restrictions on $\phi$:
    
    \begin{enumerate}[label=(\alph*)]
        \item \label{rest:i} The screw numbers are all negative. 
        \item \label{rest:ii} It leaves at least one boundary component 
        pointwise fixed 
        and the fractional Dehn twist coefficients at these boundary components 
        are 
        positive.
    \end{enumerate}
    
    Denote by $\partial^1 \Si \subset \partial \Si$ the union of the boundary 
    components pointwise fixed by $\phi$. We assume that $\phi$ is given in 
    some almost-canonical form as in \Cref{rem:mod_canonical_form}.

    \begin{notation}\label{not:quot_graph}
        We define a graph $G(\phi)$ associated to a given almost-canonical form:
        
        \begin{enumerate}
            \item It has a vertex $v$ for each subsurface of $\Si \backslash 
            \mathcal{A}$ whose connected components are cyclically permuted by 
            $\phi$.
            
            \item For each set of annuli in $\calA$ permuted cyclically it has 
            an edge 
            connecting the vertices corresponding to the surfaces on each side 
            of the collection of annuli.
        \end{enumerate}
        Let $ \NNd$ denote the set of vertices of $G(\phi)$. 
    \end{notation}
    
    \begin{definition}\label{def:filtering_function}
        We say that a function $L: \NNd \to \ZZ_{\geq  0}$ is a
        \textit{filtering function} for $G(\phi)$ if it satisfies:
        \begin{enumerate}
            
            \item \label{it:filt_i}
            If $v,v' \in  \NNd$ are connected by an edge, then $L(v) \neq 
            L(v')$.
            
            \item \label{it:filt_ii}
            If $v\in \NNd$, then either $v$ has a neighbor $v'\in \NNd$ with 
            $L(v) > L(v')$, or $L(v)=0$ and $\Si_v$ contains a component of 
            $\partial^1\Si$.
        \end{enumerate}
    \end{definition}
    
    Condition $ii)$ above implies that for $L$ to be a filtering function,
    $L^{-1}(0)$ must only contain vertices corresponding to subsurfaces of 
    $\Si \backslash \mathcal{A}$ that contain some component of $\partial^1 
    \Si$. That same condition assures us that $L^{-1}(0)$ is non-empty.
    
    \begin{definition}\label{def:distance_function}
        Define the function $ D: \NNd \to \ZZ_{\geq  0}$ as follows:
        \begin{enumerate}
            \item $ D(v) = 0$ for all $v$ with $\Si_v \cap \partial^1 \Si \neq 
            \emptyset$.
            \item $ D(v)$ is the distance to the set $ D^{-1}(0)$, that is the 
            number of edges of the smallest bamboo in $G(\phi)$  connecting $v$ 
            with some vertex in $ D^{-1}(0)$.
        \end{enumerate}
    We call it  the {\em distance function} or {\em distance to the boundary}.
    \end{definition}
    
    \begin{remark}\label{rem:modified_canonical}
        Take some $\phi: \Si \to \Si$ in canonical form and observe that the 
        function $D$ might not be a filtering function. It can happen that 
        there are two adjacent vertices $v,v'\in \NNd$ with $ D(v)= D(v')$ or 
        even that there is a vertex with a loop (an edge starting and ending at the same vertex) based at it. See for example
        \Cref{ex:realization}.  However we  modify the canonical form into an 
        almost-canonical form for which the function $ D$ {\em is} a filtering 
        function:
        
        Let $\phi: \Si \to \Si$ be an automorphism in canonical form such that $
        D(v)=  D(v')$ for some adjacent $v,v' \in \NNd$. Take one edge 
        joining $v$ and $v'$, this edges corresponds to a set of annuli 
        $\calA_1, \ldots, \calA_k$ being permuted cyclically by $\phi$. For 
        each $i=1, \ldots, k$, let $\eta_i: \MS^1 \to \calA_i$ be 
        parametrizations as in \Cref{lem:linear_ex}. Let $\calC_i \subset 
        \calA_i$ be the core curves of the annuli. We distinguish two cases:
        
        \begin{enumerate}
            \item[(1)] The core curves are not amphidrome. By 
            \Cref{rem:linear_prep} we can 
            isotope $\phi$ on the annuli $\calA_i$ to a automorphism 
            $\tilde{\phi}$ 
            without changing the action of $\phi$ on $\partial \calA_i$ so that 
            in the annuli $\eta_i(\MS^1 \times [\frac{1}{3}, \frac{2}{3}])$ it 
            is periodic. In doing so, we can redefine the canonical form to an 
            almost-canonical form as follows.
            \begin{enumerate}
                \item for each $i=1, \ldots,k$ take $\calC_i$ out from the set 
                $\calC$ and 
                include $\eta_i(\MS^1 \times \{ \frac{1}{6} \})$ and  
                $\eta_i(\MS^1 \times 
                \{\frac{5}{6} \})$.
                \item for each $i=1, \ldots,k$ take $\calA_i$ out of $\calA$ 
                and include 
                $\eta_i(\MS^1 \times [0,1/3])$ and  $\eta_i(\MS^1 \times 
                [\frac{2}{3},1]).$
            \end{enumerate}
            It is clear that this new set of data defines an almost canonical 
            form for 
            $\tilde\phi$ and that on the corresponding $G(\tilde{\phi})$ the 
            vertices $v$ and $v'$ are no longer adjacent since a new vertex 
            corresponding to the surface $\bigcup_i \eta_i(\MS^1 \times 
            [\frac{1}{3}, \frac{2}{3}] )$ appears between them.
            
            \item[(2)] The core curves are amphidrome. This case is completely 
            analogous to 
            case $(1)$ with the advantage that by definition of $\tilde{\calD}_s$
            in \Cref{not:amphi_dehn}, it is already periodic in the central 
            annuli.
        \end{enumerate}
        
        It is clear that after performing $(1)$ or $(2)$ (accordingly) for all 
        pairs of adjacent vertices $v,v'$ with $D(v)= D(v')$ we provide $\phi$ 
        with an almost-canonical whose distance function $D$ is a filtering 
        function.
    \end{remark}
    
    \begin{remark}\label{rem:not_amphidrome}
        We observe that orbits of amphidrome annuli $\calA_1, \ldots, \calA_i$ 
        correspond to loops in $G(\phi)$. So we have that after performing the 
        modification of \Cref{rem:modified_canonical}, the almost-canonical 
        form of $\phi$ does not have any amphidrome annuli in $\calA$. However, 
        some of the surfaces of $\Si \setminus \calA$ are now amphidrome annuli.
    \end{remark}
    
    \begin{notation}\label{not:qp2}
        We assume $\phi$ is in the almost-canonical form induced from the 
        canonical form after performing the modification described in 
        \Cref{rem:modified_canonical}. We denote by $\hat{\Si}$ the closure of 
        $\Si\setminus\calA$ in $\Si$. Let $\hat G(\phi)$ be a graph constructed 
        as follows:
        
        \begin{enumerate}
            \item It has a vertex for each connected component of 
            ${\Si}\setminus\calC$.
            \item There are as many edges joining two vertices as curves in 
            $\calC$ intersect the two surfaces corresponding to those vertices.
        \end{enumerate}
        We observe that the previously defined $G(\phi)$ is nothing but the 
        quotient of
        $\hat G(\phi)$ by the action induced by $\phi$ on the connected 
        components of $\Si \setminus \calA$.
        
        Let $\hat \NNd$ be the set of vertices of $\hat G(\phi)$. Since $\phi$ 
        permutes the surfaces in $\hat{\Si}$, it induces a permutation of the 
        set $\hat \NNd$ which we denote by $\sigma_\phi$. We label the set 
        $\hat \NNd$, as well as the connected components of 
        $\hat{\Si}$ and the connected components of $\calA$ in the following 
        way:
        \begin{enumerate}
            \item Label the vertices that correspond to surfaces containing 
            components of
            $\partial^1 \Si$ by $v^0_{1,1},v^0_{2,1}, \ldots, v^0_{\beta_0,1}$. 
            Let 
            $V^0$ be the union of these vertices. Note that
            $\sigma_\phi(v^0_{j,1})=v^0_{j,1}$ for all $j=1, \ldots, \beta_0$.
            \item Let $\hat D: \hat \NNd \rightarrow \ZZ_{\geq 0}$  be the 
            distance function to $V^0$, that is, $\hat D(v)$ is the number of 
            edges of the smallest path in $\hat G(\phi)$  that joins $v$ with 
            $V^0$. Let $V^i:= \hat D^{-1}(i)$. Observe that the permutation 
            $\sigma_\phi$ leaves the set $V^i$ invariant. There is a labeling 
            of $V^i$ induced by the orbits of $\sigma_\phi$: suppose it has 
            $\beta_i$ different orbits. For each $j=1, \ldots, \beta_i$, we 
            label the vertices in that orbit by $v^i_{j,k}$ with $k=1, \ldots, 
            \alpha_j$ so that $\sigma_\phi(v^i_{j,k})=v^i_{j,k+1}$ and 
            $\sigma_\phi(v^i_{j,\alpha_j})=v^i_{j,1}$. 
        \end{enumerate}
        
        Denote by $\Si^i_{j,k}$ the surface in $\hat{\Si}$ corresponding to the 
        vertex $v^i_{j,k}$. Denote by $\Si^i$ the union of the surfaces 
        corresponding to the vertices in $V^i$. We denote by $\Si^{\leq i}$ the 
        union of $\Si^0, \ldots, \Si^i$ and the annuli in between them.
        
        We recall that  $\alpha_j$ is the smallest positive number such that 
        $\phi^{\alpha_j}(\Si^i_{j,k})=\Si^i_{j,k}$.
        
    \end{notation} 
    
    \begin{thm}\label{thm:realization_mixed}
        Let $\phi: \Si \rightarrow \Si$ be a pseudo-periodic automorphism 
        satisfying assumptions \ref{rest:i} and \ref{rest:ii}. Then there 
        exists a relative mixed \tat graph 
        $(\G^{\bullet}, A^{\bullet}, \delta_\bullet)$ with $\G$ embedded in 
        $\Si$ such 
        that:
        \begin{enumerate}
            \item $\delta_i$ is a constant function for each $i=1, \ldots, d$.
            \item $[\phi]_{\partial^1 \Si}= [\phi_{\G}]_{\partial^1 \Si}$.
            \item $\phi|_{\partial\Si \setminus \partial^1 \Si}= 
            \phi_{\G}|_{\partial\Si 
                \setminus \partial^1 \Si}$.
            \item Filtration indexes are induced by the distance function $D$ 
            for the 
            almost-canonical form induced from the canonical form by 
            \Cref{rem:modified_canonical}.
        \end{enumerate}
        
    \end{thm}

    Now we state and prove \Cref{lem:tech_tat} and 
    \Cref{lem:technical_mixed_bound}
    which are used in the proof of \Cref{thm:realization_mixed}.
    
    \begin{lemma}\label{lem:tech_tat}
        Let $\phi: \Si \rightarrow \Si$ be a periodic automorphism of order 
        $n$. Let 
        $C=C_1 \sqcup \cdots \sqcup C_k$ be a non-empty collection of boundary 
        components of $\Si$ such that $\phi(C_i)=C_i$, that is, each one is 
        invariant by $\phi$. For each $i$ let $m_i$ be a metric on $C_i$ 
        invariant by $\phi$. Then there exists a relative metric ribbon graph 
        $(\G, A) \hookrightarrow 
        (\Si, \partial \Si \setminus C)$  and parametrizations of the cylinders 
        (see \Cref{not:cylinders_notation}) $r_i: \MS^1 \times I \to 
        \tilde\Si_i$ such that:
        \begin{enumerate}
            \item $\phi(\G)=\G$ and the metric of $\G$ is also invariant by 
            $\phi$.
            \item $l(\tilde\G_i) = l(C_i)$.
            \item The projection from $C_i$ to $\tilde{\G}_i$ induced by $r_i$ 
            is an 
            isometry, that is, the map $$ r(\theta,0) \mapsto r(\theta,1) $$ is 
            an isometry.
            \item $\phi$ sends retractions lines (i.e. $\{\theta\} \times I$) 
            to retractions lines.
        \end{enumerate}
    \end{lemma}
    
    \begin{proof}
        The proof uses essentially the same technique used in the proof of 
        Theorem 5.14 and 5.27 in  \cite{Bob}. For completeness we outline it 
        here.
        
        Let $\Si^{\phi}$ be the orbit surface and suppose it has genus $g$ and 
        $r \geq k$ boundary components. Let $p: \Si \to \Si^{\phi}$ be the 
        induced branch cover. 
        
        Take any relative spine $\G^{\phi}$ of $\Si^\phi$ that:
        
        \begin{enumerate}[label=(\arabic*)]
            \item Contains all branch points of the map $p$.
            \item Contains the boundary components $p(\partial \Si \setminus 
            C)$.
            \item Admits a metric such that $p(C_i)$ retracts to a part of the 
            graph of length $l(C_i)/n$.
        \end{enumerate}
        
        We observe that conditions $(1)$ are $(2)$ are trivial to get. Condition 
        $3)$ follows because of the proof of Theorem 5.14 and 5.27 in \cite{Bob}. 
        There, the conditions on the metric of the graph $\G^{\phi}$ come from 
        the rotation numbers of $\phi$, however, we do not use that these 
        numbers come from $\phi$ in finding the appropriate graph so exactly 
        the 
        same argument applies.
        
        Observe that since the metric on $C_i$ is invariant by $\phi$, there is 
        a metric induced on $p(C_i)$ for $i=1, \ldots k$. Now choose any 
        parametrizations (or product structures) of the cylinders in 
        $\Si^{\phi}_{\G^{\phi}}$ such that their retractions lines induce an 
        isometry from $p(C_i)$ to $\TG^\phi_i$.
        
        Define $\G:=p^{-1}(\G^{\phi})$. By construction, this graph satisfies 
        $(i)$ and 
        $(ii)$. The preimage by $p$ of retraction lines on $\Si^{\phi}$ gives 
        rise to parametrizations of the cylinders in $\Si_{\G}$ satisfying 
        $iii)$ and $iv)$.
    \end{proof}

    \begin{lemma}\label{lem:technical_mixed_bound}
        Let $(\G^{\bullet}, B^{\bullet}, \delta_\bullet)$ be a relative mixed 
        \tat graph embedded in a surface $\Si$ and let $C_1, \ldots, C_k 
        \subset B$ a set of relative boundary components cyclically permuted by 
        $\phi$. Suppose that all the vertices in these boundary components are 
        of valency $3$. Then we can  modify the metric structure of the graph 
        to produce a mixed \tat graph 
        $(\hat{\G}^{\bullet}, \hat{B}^{\bullet}, \delta_\bullet)$ with $l(C_i)$ 
        as small as we want and with $[\phi_\G]_{\partial^1 \Si} = 
        [\phi_{\hat{\G}}]_{\partial^1 \Si}$.
    \end{lemma}
    
    \begin{proof}
        Let $e_1,\ldots, e_m$ be the edges comprising $C_1$, where $e_j$ has 
        length $l_j$. Let $v_1,\ldots, v_m$ be the vertices
        of these edges, so that $e_i$ connects $v_i$ and $v_{i+1}$ (here, 
        indices are taken modulo $m$). Let $f_{i,j}$,
        for $j=1,\ldots, n_i$ be the edges adjacent to $v_i$, other than
        $e_i, e_{i+1}$, in such a way that the edges have the cyclic order
        $e_{i+1}, e_{i}, f_{i,1}, \ldots, f_{i,n_i}$. Let
        $\epsilon < l(C_1)$. We would like to replace $C_1$ with
        a circle of length $l(C_1) - \epsilon$. We assume that $l_1 = \min_i 
        l_i$.
        
        If $\epsilon/m \leq l_1$, then we do the following:
        \begin{itemize}
            
            \item
            Each edge $e_i$ is modified to have length $l_i - \epsilon/m$.
            
            \item
            For any $i$ with $n_i = 1$, the length of $f_{i,1}$ is increased by
            $\epsilon/2m$.
            
            \item
            For any $i$ with $n_i > 1$, extrude an edge $g_i$ from the vertex 
            $v_i$ of
            length $\epsilon/2m$ so that one end of $g_i$ is adjacent to
            $e_i, e_{i+1}$ and $g_i$, and the other
            is adjacent to $f_{i,1}, \ldots, f_{i,n_i}$ and $g_i$,
            with these cyclic orders.
            
        \end{itemize}
        In the case when $\epsilon/m > l_1$, we execute the above procedure with
        $\epsilon$ replaced by $m\cdot l_1$, which results in a circle made up
        of fewer edges. After finitely many steps, we
        obtain the desired length for $C_1$.
        
        \begin{figure}[H]
            \includegraphics[width=100mm]{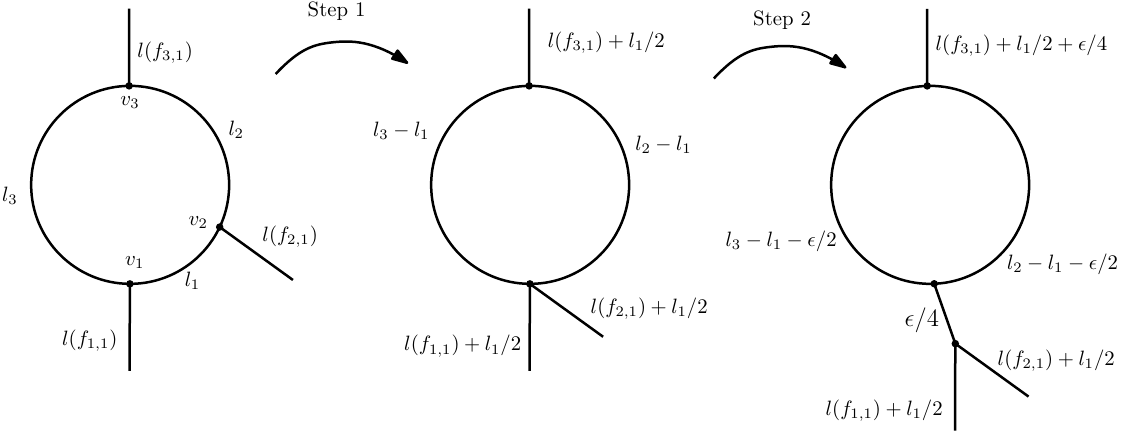}%
            \caption{Example of modification at a boundary component. Suppose 
                that
                $l_1<l_2<l_3$. In step $1$ we reduce the length of the circle 
                by $l_1$. In step
                $2$ we reduce it by $\epsilon$.}
            \label{fig:technical_remark}
        \end{figure}
        
    \end{proof}
    
    \begin{proof}[Proof of \Cref{thm:realization_mixed}]
        By definition, all surfaces corresponding to vertices in $D^{-1}(0)$ 
        are 
        connected because they are invariant by $\phi$. We have that 
        $\phi|_{\Si^0_{j,1}}:\Si^0_{j,1} \to \Si^0_{j,1}$ is periodic outside a 
        neighborhood of $\partial^1 \Si \cap \Si^0_{j,1}$ and that the 
        fractional Dehn 
        twist coefficients with respect to all the components in $\partial^1 
        \Si \cap 
        \Si^0_{j,1}$ are positive. Denote $B^0_{j,1}:= \partial \Si^{0}_{j,1} 
        \setminus 
        \partial^1 \Si$. By \Cref{thm:characterization_tat}, for each 
        $j=1,\ldots, 
        \beta_0$, there is a relative \tat graph $(\G^0_{j,1},B^0_{j,1})$ 
        embedded in 
        $\Si^0_{j,1}$ modeling $\phi|_{\G^0_{j,1}}$. Denote $\G[0]:= 
        \bigsqcup_{j} 
        \G^0_{j,1}$ and  $B[0]:=\bigsqcup_{j} B^0_{j,1}$. Then  $(\G[0], 
        B[0], \delta_0)$ is a relative mixed \tat graph of depth $0$ for 
        $\Si^{\leq 0}$ 
        (it is just a relative \tat graph) such that:
        \begin{enumerate}
            \item $\delta_0|_{\Si^0_{j,1}} = \pi$ for all $j= 1, \ldots, 
            \beta_0$.
            \item $[\phi|_{\Si^0}]_{\partial^1 \Si^0} = 
            [\phi_{\G[0]}]_{\partial^1 \Si^0}$. 
            \item $\phi|_{B[0]} = \phi_{\G[0]}|_{B[0]}$.
            \item All the vertices on $B[0]$ have valency $3$.
        \end{enumerate}
        
        Suppose that we have a relative mixed \tat graph $(\G[a-1]^\bullet, 
        B[a-1]^\bullet, \delta[a-1]_{\bullet})$ of depth $a$ embedded as a 
        spine in
        $\Si^{\leq a}$ 
        and with $B[a-1] = \partial \Si^{\leq 
            a-1} \setminus \partial^1 \Si$ such that:
        
        \begin{enumerate}
            \item \label{it:hyp_i} $\delta[a-1]_i$ is a constant function for 
            each $i=0, 
            \ldots, a -1$
            \item \label{it:hyp_ii} $[\phi|_{\Si^{\leq a -1}}]_{\partial^1 
                \Si^{\leq a- 1}} 
            = [\phi_{\G[a-1]}]_{\partial^1 \Si^{\leq a-1}}$
            \item \label{it:hyp_iii} $\phi|_{B[a-1]} = 
            \phi_{\G[a-1]}|_{B[a-1]}$.
            \item \label{it:hyp_iv} All the vertices on $B[a-1]$ have valency 
            $3$.
        \end{enumerate}
        We recall that $\phi_{\G[a-1]}$ denotes the mixed \tat automorphism 
        induced by 
        $(\G[a-1]^\bullet, B[a-1]^\bullet, \delta[a-1]_{\bullet})$. We extend 
        $\G[a-1]$ to a mixed \tat graph $\G[a]$ satisfying 
        \ref{it:hyp_i} - \ref{it:hyp_iv}. This proves the theorem by induction. 
        We focus on a particular orbit of surfaces. Fix  $j\in \{1, \ldots, 
        \beta_a\}$ and consider the surfaces $\Si^a_{j,1}, \ldots, 
        \Si^a_{j,\alpha_j} \subset \Si^a$ with $\phi(\Si^a_{j,k}) = 
        \Si^a_{j,k+1}$  and $\phi(\Si^a_{j,\alpha_j}) =\Si^a_{j,1}$. 
        
        For each $j$, we distinguish two types of boundary 
        components in the orbit $\bigsqcup_k \Si^a_{j,k}$:
        \begin{enumerate}
            \item[Type I)]  Boundary components that are connected to an 
            annulus whose other
            end is in $\Si^{a-1}$, we denote these by $\partial^{I}$.
            \item[Type II)] The rest: boundary components that are in $\partial 
            \Si$ and 
            boundary components that are connected to an annulus whose other 
            end is in 
            $\Si^{a+1}$, we denote these by $\partial^{II}$.
        \end{enumerate}
        
        Since we are doing the construction for an orbit, we use \textit{local 
 notation} in which not all the indices are specified so that the formulae is 
        easier to read.
        
        Let $\calA^{I}$ denote the union of annuli connected to boundary 
        components in $\partial^{I}$. These annuli are permuted by $\phi$. 
        Suppose that 
        there are $r'$ different orbits of annuli $\calA_1, \ldots 
        \calA_{r'}$, and let $\ell_i\in \NN$ be the length of the orbit 
        $\calA_i$. Let $s_i$ be the screw number of the orbit $\calA_i$ 
        (recall \Cref{def:screw_number} and \Cref{rem:screw_indep}). Let 
        $\calB_{i,1},
        \ldots, \calB_{i, \ell_i}$ 
        be the orbit of boundary components of $\Si^{a-1}$ that are contained 
        in the 
        orbit $\calA_i$. The metric of $\G[a-1]$ gives lengths to these 
        boundary 
        components and all the boundary components in the same orbit have the 
        same 
        length $l(\calB_{i,1}) \in \RR_{+}$. Consider the positive real numbers 
        \begin{equation}\label{eq:screw_n}
            \frac{s_1}{\ell_1} l(\calB_{1,1}), \ldots, \frac{s_{r'}}{\ell_{r'}} 
            l(\calB_{r',1})
        \end{equation}
        Using \Cref{lem:technical_mixed_bound}, we modify the metric structure 
        of 
        $\G[a-1]$ near each orbit $\calB_i$ so that $$\frac{s_1}{\ell_1} 
        l(\calB_{1,1})= \cdots = \frac{s_{r'}}{\ell_{r'}} l(\calB_{r',1}).$$ 
        This 
        is possible since we can make $l(\calB_{i,1})$ as small as needed.
        
        For each $i=1, \ldots, r'$, let $\calA_{i,1}, \ldots 
        \calA_{i,\ell_i}$ be the annuli in the orbit $\calA_i$ and let 
        $\calB'_{i,1}, \ldots \calB'_{i, \ell_i}$ be the boundary components 
        that they 
        share with $\Si^{a}$. Consider parametrizations $\eta_{i,1}, \ldots, 
        \eta_{\ell_i,1}$ given by \Cref{lem:linear_ex}. The metric on the 
        boundary 
        components of $B[a-1]$ and the parametrizations induce a metric on all 
        the 
        boundary components in $\partial^{I}$ that is invariant by $\phi$. 
        
        We observe that $\phi^{\alpha_j}|_{\Si^a_{j,1}}:\Si^a_{j,1} \to 
        \Si^a_{j,1}$ is 
        periodic and $\partial^{I} \cap \Si^{a}_{j,1}$ is a subset of boundary
        components 
        that have a metric. So we can apply \Cref{lem:tech_tat} and we get a 
        relative 
        metric ribbon graph $(\G^{a}_{j,1}, \partial^{II} \cap \G^{a}_{j,1})$ 
        and 
        parametrizations  of each cylinder in $(\Si^a_{j,1})_{\G^{a}_{j,1}}$ 
        with 
        properties $i), \ldots, iv)$ in the the Lemma. We can translate this 
        construction by $\phi$ to  the rest of the surfaces $\Si^a_{j,2}, 
        \ldots, 
        \Si^a_{j,\alpha_j}$. So we get graphs $\G^a_{j,k} 
        \hookrightarrow \Si^{a}_{j,k}$ and parametrizations for the cylinders 
        in $(\Si^a_{j,k})_{\G^{a}_{j,k}}$ for all $k=1, \ldots, \alpha_j$. The 
        construction assures us that $\phi|_{\Si^{a}_{j,\alpha_j}}: 
        \Si^{a}_{j,\alpha_j} \to 
        \Si^{a}_{j,1}$ sends $\G^a_{j,\alpha_j}$ to $\G^a_{j,1}$ isometrically 
        and that it takes retractions lines of the parametrizations in 
        $\Si^a_{j,\alpha_j}$ to retraction lines  in $\Si^a_{j,1}.$
        
        We proceed to extend $\G[a-1]$ to the orbit of $\Si^a_{j,1}$. For each 
        $i=1, \ldots, r'$ do the following:
        \begin{enumerate}
            \item[Step 1.] Remove $\calB_{i,1}, \ldots, \calB_{i,\ell_i}$ from 
            $\G[a-1]$.
            \item[Step 2.] Take $\epsilon>0$ small enough. Decrease by 
            $\epsilon$ the metric
            on all the 
            edges of $\G[a-1]$ adjacent to vertices in  $\calB_{i,1}, \ldots, 
            \calB_{i,\ell_i}$.
            \item[Step 3.] Add to the graph the retraction lines of the 
            parametrizations 
            $\eta_{i,1}, \ldots, \eta_{i, \ell_i}$ that were adjacent to 
            vertices in  
            $\calB_{i,1}, \ldots, \calB_{i,\ell_i}$. That is, if $v\in 
            \calB_{i,1} 
            \subset \calA_{i,1}$ include $\eta_{i,1}(\{v\} \times I)$. Define 
            the length of
            these segments as $\epsilon/2$.
            \item[Step 4.] Add to the graph the retraction lines of the 
            parametrizations of 
            the cylinders $(\Si^a_{j,k})_{\G^a_{j,k}}$ that start at the ends 
            of the lines 
            added in the previous step. Define the length of these segments as 
            $\epsilon/2$.
            \item[Step 5.] Add to the graph the graphs $\G^a_{j,1}, \ldots, 
            \G^a_{j,
                \alpha_j}$.
        \end{enumerate}
        
        We repeat this process for 
        all orbits of surfaces in $\Si^a$ and so we extend the graph $\G[a-1]$ 
        to all $\Si^a$. Denote $$\G[a]^a:= \bigsqcup_{j,k} \G^a_{j,k}.$$ Denote 
        the resulting graph by $\G[a]$.
        
        We make the following 
        observation: $(\G[a]_{\G[a]^a}, \TG[a]^a)$ is by construction isometric 
        to $(\G[a-1], \calB[a-1])$. We denote the induced relative mixed \tat 
        automorphism by $\phi_{\G[a]_{\G[a]^a}}$ which acts on $\Si^{\leq 
        a}_{\G[a]^a}$.
        By the previous observation there is an induced filtration on $\G[a]$: 
        $$\G[a] = \G[a]^0 \supset \G[a]^1 \supset \cdots \supset \G[a]^{a-1} 
        \supset \G[a]^a$$ and similarly for the relative parts. We define 
        $\delta_a: \G[a]^a \to\RR_\geq0$ to be the constant function equal to 
        the numbers \cref{eq:screw_n}.
        (which are by construction the same number).
        
        By the choice of $\delta_a$ and the parametrizations on the annuli that 
        join
        $\Si^{a-1}$ with $\Si^{a}$ we have that  $$\calD_{\delta_a}\circ
        \phi_{\G_{\G[a]}}: \Si_{\G[a]^a} \rightarrow \Si_{\G[a]^a}$$ is 
        compatible with
        the gluing $g_{a+1}$. So that $(\G[a], \calB[a])$ is a relative mixed 
        \tat graph follows from $I')$ in \Cref{rem:equiv_mixed_tat}.

        We have already made sure in the construction that \ref{it:hyp_i} and
        \ref{it:hyp_iv} hold in $\G[a]$. 
        
        Let's show that \ref{it:hyp_ii} and \ref{it:hyp_iii} also
        hold. Observe that by construction $\phi$ leaves $\G[a]$ invariant so 
        there is an automorphism $\tilde{\phi}_a: \Si_{\G[a]^a} \to 
        \Si_{\G[a]^a}$ induced. This automorphism  coincides with 
        $\calD_{\delta_a}\circ\phi_{\G_{\G[a]}}$ on $\TG[a]$ by the choice of 
        the parametrizations of the annuli $\calA$ and by the choice of the 
        number $\delta_a$. Also, by the choice of $\delta_a$ and \cite[Remark 
        7.28]{Bob} we see that they have the same screw numbers on the annuli 
        connecting the level $a-1$ and the level $a$. From this discussion we 
        get \ref{it:hyp_ii} and \ref{it:hyp_iii} and finish the proof.   
    \end{proof}
    
    \begin{remark}
        From the proof we get as an important consequence that a more 
        restrictive definition of a mixed \tat graph is valid: it is enough to 
        consider mixed \tat graphs where $\delta_i$  is a constant function 
        (i.e. a number) for all $i=0, \ldots, \ell$.
    \end{remark}

    \begin{cor}
        The monodromy associated with a reduced holomorphic function germ 
        defined on an isolated surface singularity is a mixed \tat twist.
        Conversely, let $C(\G)$ be the cone over the open book associated with 
        a  mixed \tat graph. Then there exists a complex structure on $C(\G)$ 
        and  
        a reduced holomorphic function germ $f:C(\G) \to \C$ inducing $\phi_\G$ 
        as the monodromy of its Milnor fibration.
    \end{cor}

    \begin{proof}
    
    The statement follows from \Cref{thm:realization_mixed} 
    and \cite[Theorem 2.1]{NeumPich}. 
    \end{proof}

    \begin{example}\label{ex:realization}
        
        Let $\Si$ be the surface of \Cref{fig:realization_thm_f}. Suppose it is 
        embedded in $\RR^3$ with its boundary component being the unit circle 
        in the 
        $xy$-plane. 
        Consider the rotation of $\pi$ radians around the $z$-axis and denote 
        it by  $R_{\pi}$. By the symmetric embedding of the surface, it leaves 
        the surface  invariant. Isotope the rotation so that it is the identity 
        on $z\leq 0$ and it has fractional Dehn twist coefficient equal to 
        $1/2$. We denote the isotoped  automorphism by $T$.
        
        \begin{figure}[H]
            \includegraphics[width=100mm]{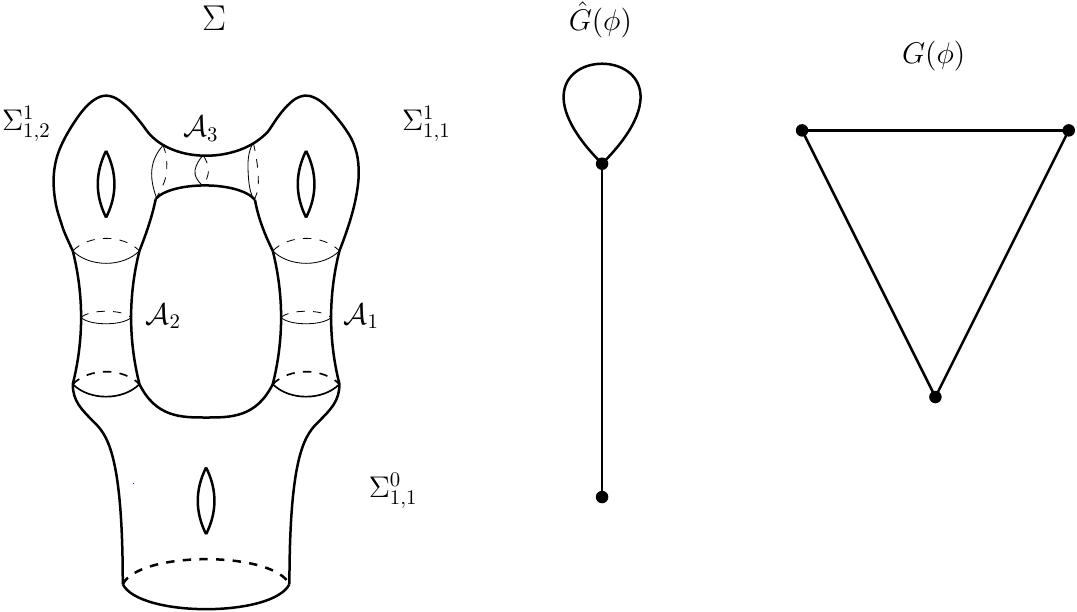}%
            \caption{On the left we see the surface $\Si$. On the right we see 
                the 
                corresponding graphs $\hat{G}(\phi)$ and $G(\phi)$ for the 
                depicted 
                canonical form.}
            \label{fig:realization_thm_f}
        \end{figure}

        More concretely, let $T:\RR^3 \to \RR^3$ defined by 
        \begin{equation}
        T(r,\theta,z):=\left\{
        \begin{array}{ccc}
        (r e^{i(\theta + \pi)},z)&\mbox{if}& z \leq \epsilon\\ (r e^{\theta + 
            \frac{z}{\epsilon}\pi},z)& \mbox{if} & 0 \leq z\leq \epsilon \\ id 
        & \mbox{if} &
        
        z \leq 0
        \end{array}
        \right.
        \end{equation}
        
        With $(r,\theta)$ polar coordinates on the $xy$-plane and $\epsilon > 
        0$ small.
        
        Let $D_i$ be a full positive Dehn twist on the annuli $\calA_i$, 
        $k=1,2,3$ (See
        \Cref{fig:realization_thm_f}). We define the 
        automorphism $$\phi:= \calD_3^{-1} \circ \calD_2^{-1} \circ
        T|_{\Si}.$$ 
        
        The automorphism comes already in canonical form. We construct the 
        corresponding Nielsen graph $\hat{G}(\phi)$ and we observe that the 
        corresponding  distance function  $D$ is not a  filtering function 
        since there is $1$ loop on $\hat{G}(\phi)$. So we apply 
        \Cref{rem:modified_canonical} and we get the almost-canonical form and 
        graphs of figure \Cref{fig:realization_thm_mod}.
        
        \begin{figure}[H]
            \includegraphics[width=100mm]{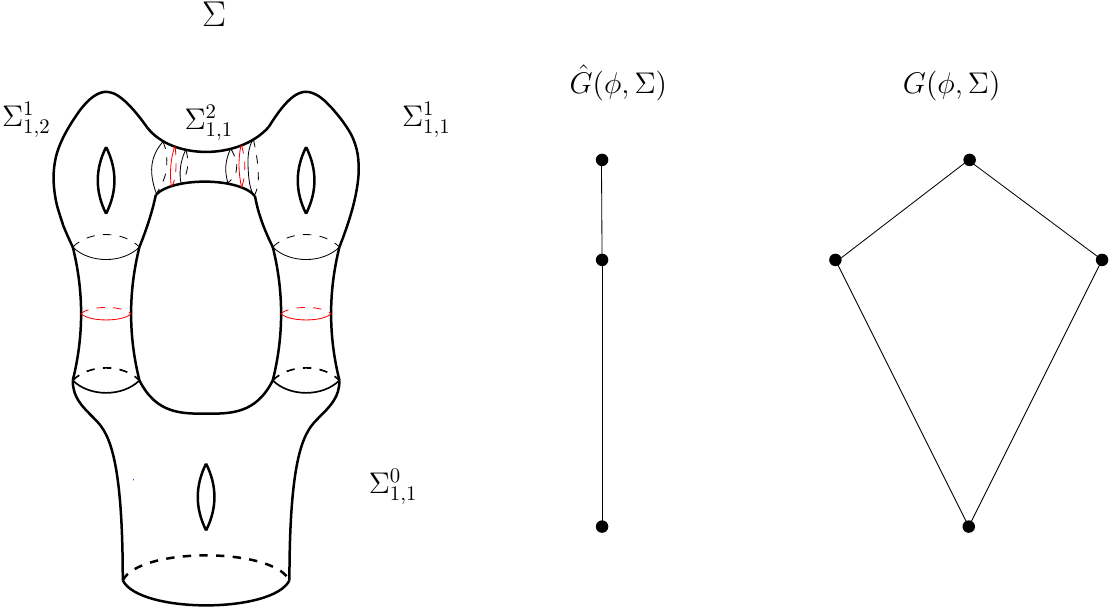}%
            \caption{On the left we see the surface $\Si$. On the right we see 
                the 
                corresponding graphs $\hat{G}(\phi)$ and $G(\phi)$ for the 
                depicted 
                almost-canonical form. In red we see the core curves of the 
                annuli in $\calA$.}
            \label{fig:realization_thm_mod}
        \end{figure}
        
        Now there are $4$ annuli in $\calA$ in this almost-canonical form. The 
        annuli
        $\calA_1$ and $\calA_2$ are exchanged by the monodromy, the Dehn twists
        $\calD_1^{-2}$ and $\calD_2$ indicate that the screw number of this 
        orbit is
        $-1$. And the annuli $\calA_{3,1}$ and $\calA_{3,2}$ that were 
        originally
        contained in $\calA_3$; these annuli are also exchanged by the 
        monodromy. We get
        that the screw  number on this orbit is $-1$.
        
        We start the construction process following 
        \Cref{thm:realization_mixed}. We
        construct a relative \tat graph $(\G[0], B[0])$ for 
        $\phi|_{\Si^0}:\Si^0 \to
        \Si^0$. We use  \cite[Theorem 5.27]{Bob} for this. In
        \Cref{fig:realization_thm_0} we can see this graph in blue.
        
        \begin{figure}[H]
            \includegraphics[width=45mm]{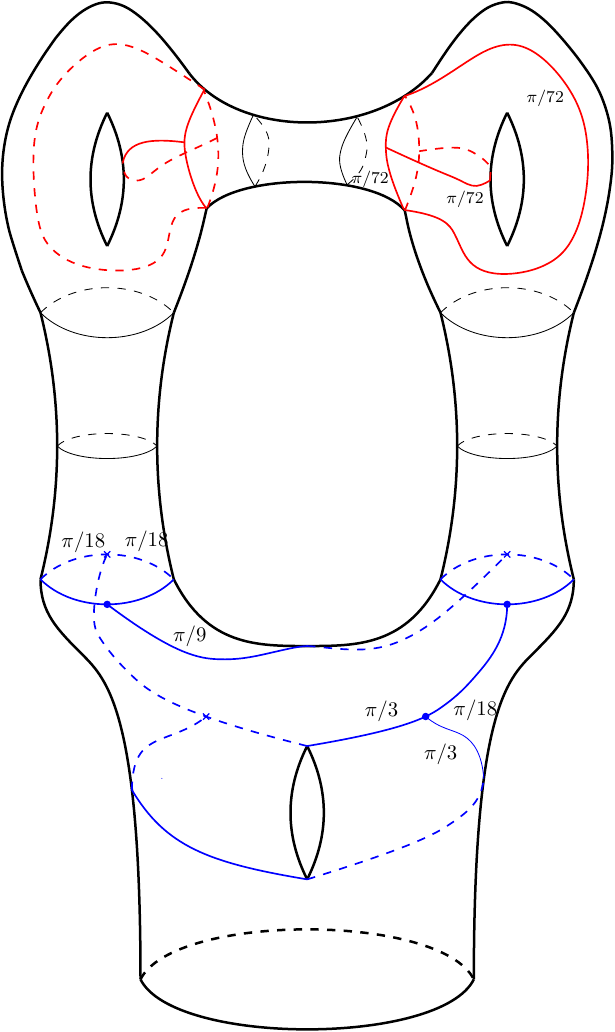}%
            \caption{The relative \tat graph $(\G[0], \B[0])$ embedded in 
                $\Si^0 \subset
                \Si$. The lengths are indicated on a few edges and the rest is 
                obtained by
                symmetry of the graph. We can also see in red an invariant 
                relative spine for
                $\Si^1$.} %
            \label{fig:realization_thm_0}
        \end{figure}
        
        In the next step we construct relative metric ribbon graphs for 
        $\Si^1$. In this
        case $\Si^1$ consist of two connected surfaces that are exchanged. Each 
        surface
        is a torus with two disks removed and one of the boundary components is 
        glued to
        an annulus connecting it with $\Si^2$. This graph will correspond with
        $\G^1_{\G^2}$ in the final mixed \tat graph. In the notation of the 
        Theorem we
        are using, it is $\G[1]^1$. In \Cref{fig:realization_thm_0} we can see 
        these
        relative metric ribbon graphs in red. In this step we also choose an 
        invariant
        product structure on $\Si^1_{\G^1}$
        
        Now we proceed to find the parametrizations $\eta_1$ and $\eta_2$. We 
        pick any
        parametrization $\eta_1$ for $\calA_1$. On the right part of
        \Cref{fig:example_thm_param} we can see the two retraction lines of the
        parametrization starting at the two vertices $p$ and $q$ of the 
        corresponding
        boundary component in $B[0].$  On the left part of that figure we see 
        two
        annuli, the upper one shows the  image of the two retraction lines by 
        $\phi$, on
        the lower annulus we see the retraction lines that we choose according 
        to
        \Cref{lem:linear_ex}.
        
        \begin{figure}[ht]
            \includegraphics[width=50mm]{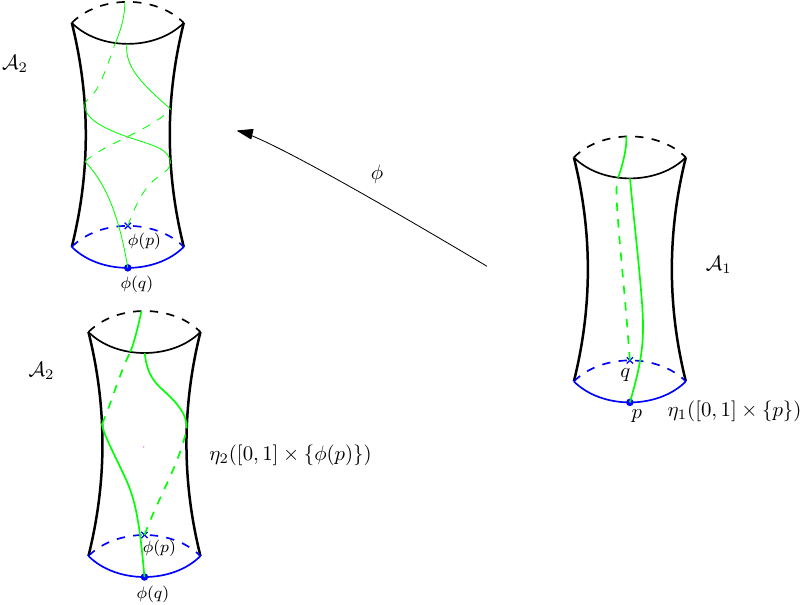}%
            \caption{The orbit of annuli $\calA_1$ and $\calA_2$. On the right 
                part we see
                $\calA_1$ with a chosen product structure and on the left part 
                we see two copies
                of $\calA_2$, the lower one with the parametrization given by
                \Cref{lem:linear_ex}.} %
            \label{fig:example_thm_param}
        \end{figure}
        
        We concatenate the chosen retraction lines of the annuli (green in
        \ref{fig:realization_thm_1}) with the corresponding retraction lines of 
        the
        product structure in $\Si^1$ (orange in the the picture).
        
        The metric on the red part is chosen so that each of the two components 
        of
        $\TG^1$ has the same length as the two relative components in 
        $\calB[0]$, that
        is equal to $\pi/9$. Since the screw number is $-1$ and there are two 
        annuli in
        the orbit we get that $\delta_1$ is the constant function $\pi/18.$
        
        \begin{figure}[ht]
            \includegraphics[width=45mm]{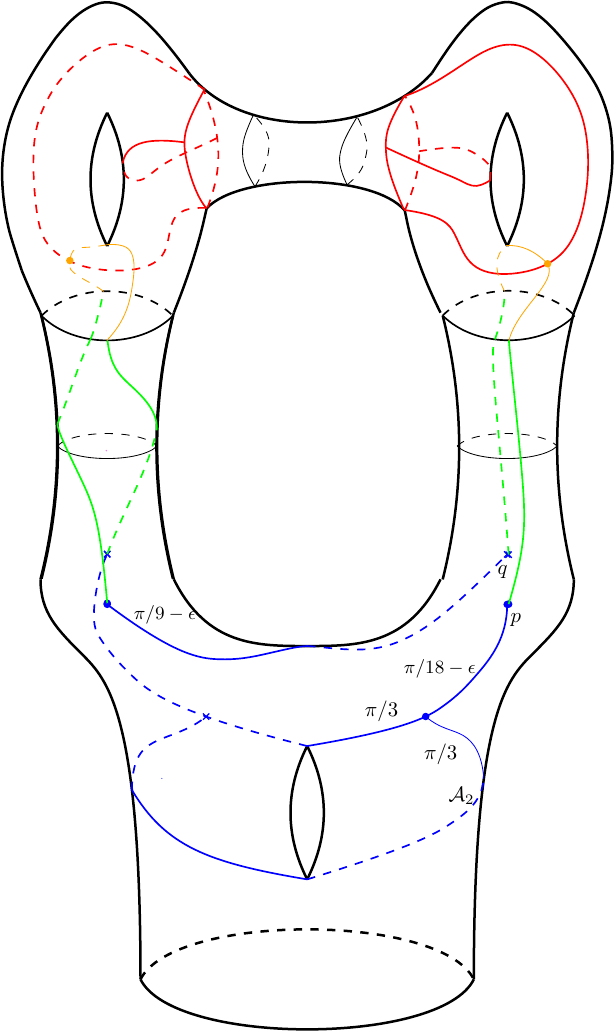}%
            \caption{The relative \tat graph $(\G[1], \B[1])$ embedded in 
                $\Si^1 \subset
                \Si$. The lengths are indicated. The green lines correspond to 
                retraction lines
                of the corresponding product structures $\eta_1$ and $\eta_2$ 
                and the orange
                lines correspond to retraction lines of the product structures 
                chosen for the
                cylinders $\Si^2_{\G[1]^1}$} %
            \label{fig:realization_thm_1}
        \end{figure}
        
        Similarly to the construction of the graph $(\G[1], \B[1])$, we 
        construct $\G[2]
        = \G$. We observe that $\Si^2$ is an annulus whose boundary components 
        are
        permuted by $\phi.$ It is attached along an orbit of annuli 
        $\calA_{3,1}$ and
        $\calA_{3,2}$ to $\Si^2$. It has total length equal to $4 \pi/72= 
        \pi/18$ Since
        this orbit of annuli has screw number $-1/2$, we get that $\delta_2$ is 
        the
        constant function $\pi/144.$
        
        \begin{figure}[ht]
            \includegraphics[width=45mm]{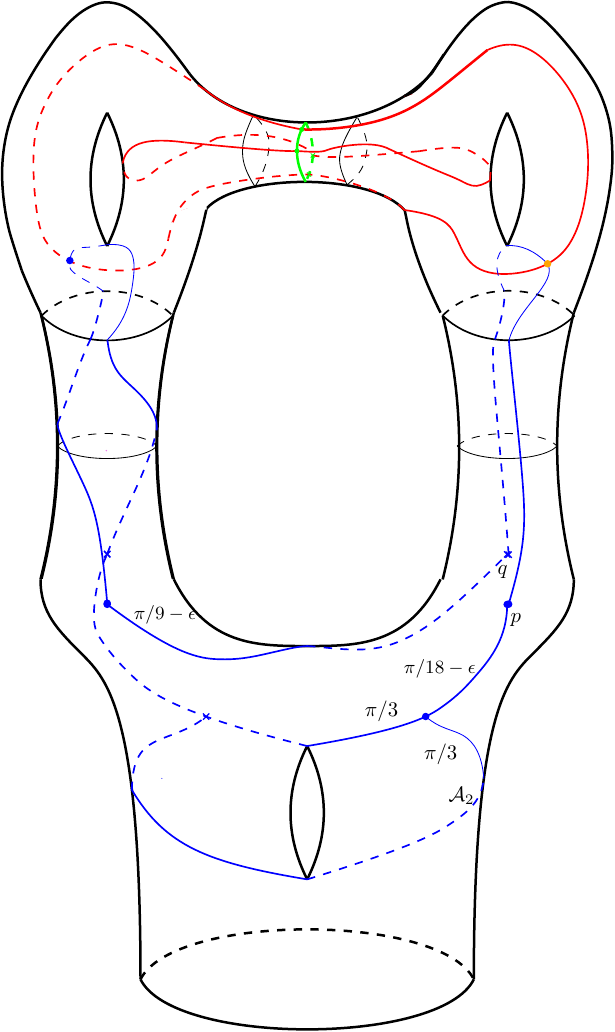}%
            \caption{The final mixed \tat graph $(\G^\bullet, \delta_\bullet)$. 
                In green we
                have $\G^2$; in red we have $\G^1 \setminus \G^2$ and in blue 
                $\G \setminus
                \G^1$.} %
            \label{fig:realization_thm_2}
        \end{figure}
    \end{example}

    \bibliographystyle{alpha}
    \bibliography{bibliography}
    
\end{document}